\documentclass[a4paper,twoside,12pt]{article}
\usepackage[a4paper,margin=3cm,centering,nohead,ignorefoot,footskip=5ex]{geometry}
\usepackage{amsmath,amssymb,amsthm,pstricks}

\newtheorem{theorem}{Theorem}
\newtheorem{definition}[theorem]{Definition}

\newcommand*{\cl}{\mathsf{CL}}
\newcommand*{\il}{\mathsf{IL}}
\newcommand*{\nf}{\mathsf{NF}}
\newcommand*{\kolmogorovtext}{K}
\newcommand*{\kolmogorov}[1]{#1^\kolmogorovtext}
\newcommand*{\ntmtext}{M}
\newcommand*{\ntm}[1]{#1^\ntmtext}
\newcommand*{\ntntext}{N}
\newcommand*{\ntn}[1]{#1^\ntntext}
\newcommand*{\im}{\operatorname{im}}
\newcommand*{\f}{F}

\begin{document}

\title{Copies of classical logic in intuitionistic logic}
\author{Jaime Gaspar\thanks{INRIA Paris-Rocquencourt, $\pi r^2$, Univ Paris Diderot, Sorbonne Paris Cit\'e, F-78153 Le Chesnay, France. \texttt{mail@jaimegaspar.com}, \texttt{www.jaimegaspar.com}. Financially supported by the French Fondation Sciences Math\'ematiques de Paris. This article is essentially a written version of a talk given at the 14th~Congress of Logic, Methodology and Philosophy of Science (Nancy, France, 19\protect\nobreakdash--26 July 2011), reporting on results in a PhD thesis~\cite[chapter~14]{Gaspar2011} and in an article~\cite{Gaspar2012}.}}
\date{8 November 2012}
\maketitle

\begin{abstract}
  Classical logic (the logic of non-constructive mathematics) is stronger than intuitionistic logic (the logic of constructive mathematics). Despite this, there are copies of classical logic in intuitionistic logic. All copies usually found in the literature are the same. This raises the question: is the copy unique? We answer negatively by presenting three different copies.
\end{abstract}

\section{Philosophy}

\subsection{Non-constructive and constructive proofs}

Mathematicians commonly use an indirect method of proof called non-constructive proof: they prove the existence of an object without presenting (constructing) the object. However, may times they can also use a direct method of proof called constructive proof: to prove the existence of an object by presenting (constructing) the object.

\begin{definition}
  \mbox{}
  \begin{itemize}
    \item A \emph{non-constructive proof} is a proof that proves the existence of an object without presenting the object.
    \item A \emph{constructive proof} is a proof that proves the existence of an object by presenting the object.
  \end{itemize}
\end{definition}

From a logical point of view, a non-constructive proof uses the law of excluded middle while a constructive proof does not use the law of excluded middle.

\begin{definition}
  The \emph{law of excluded middle} is the assertion ``every statement is true or false''.
\end{definition}

To illustrate this discussion, let us see the usual example of a theorem with non-constructive and constructive proofs.

\begin{theorem}
  There are irrational numbers $x$ and $y$ such that $x^y$ is a rational number.
\end{theorem}

\begin{proof}[Non-constructive proof]
  By the law of excluded middle, $\sqrt 2^{\sqrt 2}$ is a rational number or an irrational number.
  \begin{description}
    \item[Case $\sqrt 2^{\sqrt 2}$ is a rational number] Let $x = \sqrt 2$ and $y = \sqrt 2$. Then $x$ and $y$ are irrational numbers such that $x^y = \sqrt 2^{\sqrt 2}$ is a rational number.
    \item[Case $\sqrt 2^{\sqrt 2}$ is an irrational number] Let $x = \sqrt 2^{\sqrt 2}$ and $y = \sqrt 2$. Then $x$ and $y$ are irrational numbers such that $x^y = 2$ is a rational number.\qedhere
  \end{description}
\end{proof}

Note that the above proof is non-constructive because the proof does not present $x$ and $y$ since the proof does not decide which case holds true. Also note that the proof uses the law of excluded middle.

\begin{proof}[Constructive proof]
  Let $x = \sqrt 2^{\sqrt 2}$ and $y = \sqrt 2$. Then $x$ (by the Gelfond-Schneider theorem) and $y$ are irrational numbers such that $x^y = 2$ is a rational number.
\end{proof}

Note that the above proof is constructive because the proof presents $x$ and $y$. Also note that the proof does not use the law of excluded middle.

\subsection{Constructivism}

We saw that mathematicians use both non-constructive and constructive proofs. There is a school of thought in philosophy of mathematics, called constructivism, which rejects non-constructive proofs in favour of constructive proofs.

\begin{definition}
  \emph{Constructivism} is the philosophy of mathematics that insists on constructive proofs.
\end{definition}

Let us see some motivations for constructivism.
\begin{description}
  \item[Philosophical motivations]\mbox{}
  \begin{itemize}
    \item The more radical constructivists simply consider non-constructive proofs unsound. The less radical constructivists consider that non-constructive proofs may be sound, but not as sound as constructive proofs.
    \item Some constructivists reject the mind-independent nature of mathematical objects. So for a mathematician to prove the existence of an object, he/she has to give existence to the object by constructing the object in his/her mind.
    \item Non-constructivism puts the emphasis on truth (as in ``every statement is true or false''), while constructivism puts the emphasis on justification (as in ``we have a justification to believe that a statement is true, or we have a justification to believe that the statement is false''). Given an arbitrary statement, in general there is no justification to believe that the statement is true and no justification to believe that the statement is false, so a constructivist would not assert ``every statement is true or false'', that is a constructivist rejects the law of excluded middle.
    \item Non-constructivism does not differentiate between the quantifications $\neg \forall x \neg$ and $\exists x$, but constructivism is more refined because it differentiates between them:
    \begin{itemize}
      \item $\neg\forall x\neg$ means the usual ``there exists an $x$''; 
      \item $\exists x$ has the stronger meaning of ``there exists an $x$ and we know $x$''.
    \end{itemize}
  \end{itemize}
  \item[Mathematical motivations]\mbox{}
  \begin{itemize}
    \item Constructive proofs are more informative than non-constructive proofs because they not only prove the existence of an object, but even give us an example of such an object.
    \item We can use the constructive setting to study non-constructive principles. In the usual setting of mathematics, which includes non-constructive principles, there is no way to tell the difference between what results from the setting and what results from the non-constructive principles. But in a constructive setting we can isolate the role of non-constructive principles. For example, if we want to determine which theorems are implied by the axiom of choice, we need to do it in set theory without the axiom of choice.
   \item There are several tools in mathematical logic that work fine for constructive proofs but not for non-constructive proofs. So in order to benefit from these tools we should move to a constructive setting. For example, the extraction of computational content using G\"odel's functional interpretation can always be done for constructive proofs but has restrictions for non-constructive proofs.
  \end{itemize}
  \item[Historical motivation]\mbox{}
  \begin{itemize}
    \item Until the 19th century all proofs in mathematics were more or less constructive. Then in the second half of the 19th century there were introduced powerful, infinitary, abstracts, non-constructive principles. These principles were already polemic at the time. Even worse, at the turn of the century there were discovered paradoxes related to these non-constructive principles. Then it was not only a question of what principles are acceptable, but even the consistency of mathematics was at stake. Constructivism proposes a solution to this crisis: to restrict ourselves to the safer constructive principles, which are less likely to produce paradoxes.
  \end{itemize}
\end{description}

\section{Mathematics}

\subsection{Classical and intuitionistic logics}

We saw that non-constructivism uses the law of excluded middle while constructivism does not use the law of excluded middle. Let us now formulate this idea in terms of logic.

\begin{definition}
  \mbox{}
  \begin{itemize}
    \item \emph{Classical logic}~$\cl$ is (informally) the usual logic of mathematics including the law of excluded middle.
    \item \emph{Intuitionistic logic}~$\il$ is (informally) the usual logic of mathematics excluding the law of excluded middle.
  \end{itemize}
\end{definition}

To be sure, $\cl$ corresponds to non-constructivism, and $\il$ corresponds to constructivism.

Now let us compare $\cl$ and $\il$. We can prove the following.
\begin{itemize}
  \item $\cl$ is strictly stronger than $\il$ (that is there are theorems of $\cl$ that are not theorems of $\il$, but every theorem of $\il$ is a theorem of $\cl$).
  \item $\cl$ is non-constructive (that is there are proofs in $\cl$ that cannot be turned into constructive proofs) while $\il$ is constructive (that is every proof in $\il$ can be turned into a constructive proof).
\end{itemize}

\subsection{Copies}

To introduce the notion of a copy of classical logic in intuitionistic logic, first we need to introduce the notion of a negative translation.

\begin{definition}
  A \emph{negative translation} is a mapping $\ntntext$ of formulas that embeds $\cl$ in $\il$ in the sense of satisfying the following two conditions.
  \begin{description}
    \item[Respecting provability] For all formulas $A$ and sets $\Gamma$ of formulas we have the implication $\cl + \Gamma \vdash A \ \Rightarrow \ \il + \ntn\Gamma \vdash \ntn A$ (where $\ntn\Gamma = \{\ntn A : A \in \Gamma\}$);
    \item[Faithfulness] For all formulas $A$ we have $\cl \vdash A \leftrightarrow \ntn A$.
  \end{description}
  A \emph{copy} of classical logic in intuitionistic logic is the image $\im \ntntext$ (the set of all formulas of the form $\ntn A$) of a negative translation $\ntntext$~\cite[paragraph~14.5]{Gaspar2011} \cite[definition~1]{Gaspar2012}.
\end{definition}

Let us explain why it is fair to say that an image is a copy of classical logic in intuitionistic logic. From the definition of a negative translation we get the following equivalence:
\begin{equation*}
  \cl \vdash A \ \Leftrightarrow \ \il \vdash \ntn A.
\end{equation*}
We can read this equivalence in the following way: the formulas $\ntn A$ in $\im \ntntext$ are mirroring in $\il$ the behaviour of $\cl$. So $\im \ntntext$ is a reflection, a copy, of classical logic in intuitionistic logic.

\subsection{Question: is the copy unique?}

There are four negative translations usually found in the literature; they are due to Kolmogorov, G\"odel-Gentzen, Kuroda and Krivine. The simplest one to describe is Kolmogorov's negative translation: it simply double negates every subformula of a given formula.

All the usual negative translations give the same copy: the negative fragment.

\begin{definition}
  The \emph{negative fragment}~$\nf$ is (essentially) the set of formulas without $\vee$ and $\exists$.
\end{definition}

The fact that all the usual negative translations give the same copy leads us to ask: \emph{is the copy unique?}

Here we should mention that when we say that two copies are equal, we do not mean ``syntactically/literally equal'' (that would be too strong and easily falsified); we mean ``equal modulo $\il$'' (that is ``modulo identifying formulas that are provably equivalent in $\il$'').

\subsection{Answer: no}

In the following theorem we show that the answer to our question is \emph{no} by presenting three different copies.

\begin{theorem}
  Let us fix a formula~$\f$ such that $\cl \vdash \neg\f$ but $\il \nvdash \neg\f$ (there are such formulas~$\f$). Then
  \begin{itemize}
    \item $\nf$
    \item $\nf \vee \f = \{A \vee \f : A \in \nf\}$
    \item $\nf[\f/\bot] = \{A[\f/\bot] : A \in \nf\}$
  \end{itemize}
  are pairwise different copies~\cite[paragraph~14.10]{Gaspar2011} \cite[lemma~7, theorem~8 and proposition~9]{Gaspar2012}.
\end{theorem}

\begin{proof}[Sketch of the proof]
  We have to show the following three things.
  \begin{description}
    \item[There is an $\f$ such that $\cl \vdash \neg\f$ but $\il \nvdash \neg\f$] We can prove that $\f = \neg(\forall x \neg\neg P(x) \to \forall x P(x))$ (where $P(x)$ is a unary predicate symbol) is in the desired conditions~\cite[paragraph~14.11.6]{Gaspar2011} \cite[proof of lemma~7.1]{Gaspar2012}.
    \item[$\nf$, $\nf \vee \f$ and ${\nf[\f/\bot]}$ are copies] Let $\kolmogorovtext$ be Kolmogorov's negative translation, $\ntm A = \kolmogorov A \vee \f$ and $\ntn A = \kolmogorov A[\f/\bot]$~\cite[paragraph~14.8]{Gaspar2011} \cite[definition~6]{Gaspar2012}. We can prove that $\kolmogorovtext$, $\ntmtext$ and $\ntntext$ are negative translations (here we use the hypothesis $\cl \vdash \neg\f$) such that $\im\kolmogorovtext = \nf$, $\im\ntmtext = \nf \vee \f$ and $\im\ntntext = \nf[\f/\bot]$~\cite[paragraph~14.10]{Gaspar2011} \cite[theorem~8]{Gaspar2012}. This is pictured in figure~\ref{figure:negativetranslations}.
    \item[$\nf$, $\nf \vee \f$ and ${\nf[\f/\bot]}$ are different] We can prove that the images of two negative translations are equal if and only if the negative translations are pointwise equal (modulo~$\il$)~\cite[paragraph~14.11.4]{Gaspar2011}. And we can prove that $\kolmogorovtext$, $\ntmtext$ and $\ntntext$ are not pointwise equal by proving $\il \nvdash \ntm\bot \to \kolmogorov\bot$, $\il \nvdash \ntn\bot \to \kolmogorov\bot$ and $\il \nvdash \ntn P \to \ntm P$ (where $P$ is a nullary predicate symbol different from $\bot$) (here we use the hypothesis $\il \nvdash \neg\f$)~\cite[paragraph~14.11]{Gaspar2011} \cite[proofs of theorem~8.3 and proposition~9]{Gaspar2012}.\qedhere
  \end{description}
  \begin{figure}[htbp]
    \centering
    \begin{pspicture}(14cm,4cm)
      \psset{framearc=0.5,arrowsize=5pt}

      \psframe(0,0)(6,4) \rput(3,2){$\cl$}
      \psframe(8,0)(14,4) \rput(13.3,3.6){$\cl$}
      \psframe(8.5,0.5)(13.5,3.5) \rput(13,3){$\il$}

      \psellipse(11,2.5)(1,0.85) \rput(11,2.5){$\nf$}
      \psellipse(10.1,1.5)(1,0.85) \rput[t](10,1.6){$\nf \vee \f\vphantom{[}$}
      \psellipse(11.9,1.5)(1,0.85) \rput[t](12,1.6){$\nf[\f/\bot]$}

      \pscurve{->}(3,2.5)(7,2.85)(11,2.8) \rput(7,3.15){$\kolmogorovtext$}
      \pscurve{->}(3,1.5)(7,1.1)(10.1,1.15) \rput(7,1.4){$\ntmtext$}
      \pscurve{->}(3,1.2)(7,0.85)(11.9,1) \rput(7,0.55){$\ntntext$}
    \end{pspicture}
    \caption{the negative translations $\kolmogorovtext$, $\ntmtext$ and $\ntntext$, and the copies $\nf$, $\nf \vee \f$ and $\nf[\f/\bot]$.}
    \label{figure:negativetranslations}
  \end{figure}
\end{proof}

\bibliography{References}{}
\bibliographystyle{plain}

\end{document}